\documentclass{article}
\usepackage[utf8]{inputenc}
\usepackage{algorithm}
\usepackage{algorithmic}
\usepackage{natbib}
\usepackage{amsthm}
\usepackage{amssymb}
\usepackage{amsmath}
\usepackage{graphicx}
\newtheorem{thm}{Theorem}[section]

\newtheorem{lem}[thm]{Lemma}
\newtheorem{prop}[thm]{Proposition}
\newtheorem{defn}[thm]{Definition}

\title{A Note on John Simplex with Positive Dilation}
\author{Zhou Lu\footnote{This work is done during the author's visit at SQZ institution.} \\
Princeton University\\
\texttt{zhoul@princeton.edu}
}
\date{December 2020}

\begin{document}

\maketitle

\begin{abstract}
    We prove a John’s theorem for simplices in $R^d$ with positive dilation factor $d+2$, which improves the previously known $d^2$ upper bound. This bound is tight in view of the $d$ lower bound. Furthermore, we give an example that $d$ isn't the optimal lower bound when $d=2$. Our results answered both questions regarding John’s theorem for simplices with positive dilation raised by \cite{leme2020costly}.
\end{abstract}

\section{Introduction}
In the problem of linear function reconstruction (\cite{khachiyan1995complexity}, \cite{summa2014largest}, \cite{nikolov2015randomized}), one needs to efficiently reconstruct a linear function $f$ defined on a set $X$ by using a zeroth order oracle. When the cost of oracle querying grows along with accuracy, a John like simplex serves as a good basis as in \cite{leme2020costly}.

A crucial step of their method is to find a simplex $T$ with vertices in $X$, such that $X$ can be contained in some translate of $T$ with dilation, where a smaller (absolute value) dilation factor indicates higher efficiency. When considering negative dilation, a translate $-d T$ is able contain $X$ so the upper bound is $O(d)$ which matches its lower bound.

However, if we look at positive dilation, it seems to be less efficient than negative dilation because only a worse $d^2$ upper bound was given in \cite{leme2020costly}. Though a better bound for positive dilation does not immediately help design better algorithms for the reconstruction problem, it is natural to ask if we can improve it.

We prove a John's theorem for simplices with positive dilation factor $d+2$ which answers the question affirmatively. Furthermore, we give a counter-example that the $d$ lower bound given in \cite{leme2020costly} isn't tight when $d=2$.

\section{Background}
In the seminal work \cite{john2014extremum}, it's proven that any convex $X\subset R^d$ can be sandwiched between two concentric ellipsoids of ratio $d$. Specifically, the following theorem was given in \cite{john2014extremum}:

\begin{thm}
Given any convex body $X\subset R^d$, let $x+T B_d$ denote the minimal volume ellipsoid containing $X$, then it holds
\begin{equation}
    x+\frac{1}{d}T B_d\subset X\subset x+T B_d
\end{equation}
\end{thm}
The ellipsoid $x+T B_d$ is called 'John ellipsoid'. For a more comprehensive review we refer the readers to \cite{henk2012lowner}. In this paper we consider general $X$ which is potentially non-convex, and we aim to find a John like simplex with positive dilation factor $O(d)$ so that it can be seen as an analogue to John ellipsoid.

\begin{defn}
Given $d+1$ points $v_1,...,v_{d+1}$ in $R^d$, we call their convex hull a simplex with vertices $v_1,...,v_{d+1}$.
\end{defn}

We consider the following: given a compact set $X\subset R^d$ which is not necessarily convex or connected, we want to find $d+1$ points in $X$ forming a simplex $T$ such that $X$ is contained in some translate of $T$ with positive dilation.

A simple method to do so is by considering the maximum volume simplex (MVS) of $X$ as in \cite{leme2020costly}. 

\begin{defn}
A maximum volume simplex $T$ with vertices in a bounded set $X\subset R^d$ is one of the simplices whose euclidean measure is no less than any other simplex of $X$.
\end{defn}

\begin{prop}
There always exists a MVS as long as $X$ is compact.
\end{prop}

\begin{proof}
Because $X$ is bounded, the volume of its simplex has a least upper bound $c\ge 0$, therefore we can find a sequence of simplices $T_j$ with vertices $v_1^{(j)},...,v_{d+1}^{(j)}$, whose volumes converge to $c$. Due to the compactness of $X$, any sequence in $X$ has a converging sub-sequence which converges to a point in $X$, thus we can find indexes $a_j$ such that $v_1^{(a_j)}$ converges to a point $\tilde{v}_1\in X$.

We consider a new sequence of simplices $T_{a_j}$ whose first vertex is replaced by $\tilde{v}_1$ while others keep unchanged. It's easy to see that their volumes still converge to $c$. Again we find a converging sub-sequence of $v_2^{(a_j)}$ which converges to a point $\tilde{v}_2\in X$, and repeat this procedure until we have all $\tilde{v}_i$. Then the simplex with vertices $\tilde{v}_i$ is the desired MVS.
\end{proof}

\cite{leme2020costly} prove the following lemma by using the 'maximum volume' property:

\begin{lem}\label{lem1}
Let $T$ be the MVS with vertices in $X$, then $X$ is contained in a translate of $-d T$.
\end{lem}

Using lemma \ref{lem1} twice we get a $d^2$ upper bound for positive dilation. We are interested in proving tighter bounds for the positive dilation factor. Existing upper and lower bounds are $d^2$ and $d$ respectively as in \cite{leme2020costly}, in their paper two questions regarding John’s theorem for simplices with positive dilation were raised

\textbf{Question 1:}
Can we get a John’s theorem for simplices with positive dilation factor $O(d)$?

\textbf{Question 2:}
Given compact $X\subset R$, can we always find a triangle $T$ with vertices in $X$ so that $X$ is contained in some translate of $2T$?

We give an affirmative answer to question 1 by proving a $d+2$ upper bound, and provide a counter-example to question 2.

\section{$d+2$ upper bound}
In this section we prove a $d+2$ upper bound based on a simple observation: by fully exploiting the 'maximum volume' property, we can squeeze $-d T$ to a much smaller set which still contains $X$.

We already know that a simplex $T \in R^d$ can be covered by a translate of $-d T$, therefore a naive method (choose $T$ to be MVS and use $-d$ dilation twice) directly leads to a $d^2$ upper bound for positive dilation.

However, the first step of this method is extravagant in that we can finder a smaller set covering $X$ when $d>2$. Think about the following example in $R^3$: assume $T$ (the MVS of $X$) is the regular simplex with vertices $v_1,...,v_{4}$ and 0 as its center, then $-3 T$ has $-3 v_1,...,-3 v_{4}$ as its vertices. We denote $S_{v,i}$ to be the hyperplane parallel to hyperplane $v_1,...,v_{i-1},v_{i+1},...,v_{4}$, with $v$ lying on it. By the definition of $T$, $X$ lies between $S_{v_i,i}, S_{-\frac{5}{3}v_i,i}$ for any $i=1,...,4$, therefore much space is wasted in the 'cones' of $-3 T$.

\begin{figure}[h]
\centering
\includegraphics[width=10cm,height=8cm]{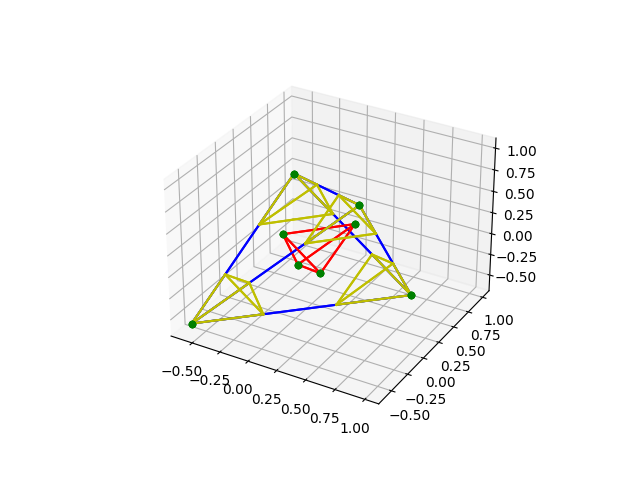}
\caption{Example in $R^3$. The red pyramid is $T$, the blue one is $-3 T$ and the yellow ones are the wasted space.}
\end{figure}

The $d^2$ bound can be decomposed as $1+(d+1)(d-1)$ where for each $v_i$, $T$ extends $d-1$ times along its direction. However, we only need to extend once which yields a $1+(d+1)=d+2$ upper bound.

\begin{thm}
For any compact set $X\in R^d$, there exists $d+1$ points in $X$ forming a simplex $T$ as its vertices, such that $X$ can be covered in a translate of $(d+2)T$. 
\end{thm}

\begin{proof}
We choose the MVS of $X$ to be $T$ with vertices $v_1,...,v_{d+1}$ and we assume 0 is its center (of gravity) without loss of generality. We denote $\hat{v}_i$ to be the symmetric point of $v_i$ with respect to hyperplane $v_1,...,v_{i-1},v_{i+1},...,v_{d+1}$ along direction $v_i$, and $S_{v,i}$ to be the hyperplane parallel to hyperplane $v_1,...,v_{i-1},v_{i+1},...,v_{d+1}$, with $v$ lying on it. By the definition of $T$, $X$ lies between $S_{v_i,i}, S_{\hat{v}_i,i}$ for any $i=1,...,d+1$, otherwise the 'maximum volume' property will be contradicted.

We take a close look at the simplex $T'$ made up of $S_{\hat{v}_i,i}$. Obviously $T'$ is a translate of $T$ with positive dilation and center (of gravity) 0 unchanged. Define the intersection point between line $v_i \hat{v}_i$ (we overload the notation of vectors to denote endpoints of a segment when the context is clear) and hyperplane $v_1,...,v_{i-1},v_{i+1},...,v_{d+1}$ as $w_i$, we have that $v_i=\frac{d}{d+1}w_i v_i$ by the definition of center (of gravity). Combing with the fact that $w_i v_i=\hat{v}_i w_i$, the dilation factor is $(1+\frac{1}{d+1})/\frac{1}{d+1}=d+2$.

Denote the simplex made up of $S_{v_i,i}$ as $\tilde{T}$ and the region enclosed by $S_{v_i,i}, S_{\hat{v}_i,i}$ as $\hat{T}$, we have the following inclusion:
\begin{equation}
    X\subset \hat{T}=T' \cap \tilde{T} \subset T'
\end{equation}
which finishes our proof.
\end{proof}

\section{Lower Bound $>d$ when $d=2$}
We give a counter-example to question 2 in this section. The idea behind is intuitive: we construct several discrete points as a hard case, so that these discrete points won't help much when we construct the covering triangle, but hurts a lot when we consider covering them since the whole convex hull needs to be covered implicitly.

\begin{thm}
There exists a compact set $X\in R^2$, such that for any triangle $T$ with vertices in $X$, $X$ can't be contained in any translate of $2T$.
\end{thm}

\begin{proof}
We pick 5 points: 
$$A=(-1,0), B=(1,0), C=(-\epsilon-\delta,1), D=(\epsilon+\delta,1), E=(0, \epsilon-1)$$
to construct $X=\{A,B,C,D,E\}$, where constants (TBD) satisfy $\epsilon, \delta \in (0,1)$. Due to symmetry, we discuss different choices of $T$ by 6 cases and how they lead to contradiction. 

\textbf{Case 1: $T= \triangle CDE$}

The intercept along $y=0$ of any translate of $2T$ has length at most $4(\epsilon+\delta)<2$, thus $AB$ can't be covered by any translate of $2T$.

\textbf{Case 2: $T= \triangle ABE$}

The intercept along $x=0$ of any translate of $2T$ has length at most $2-2\epsilon<2-\epsilon$, thus $C, E$ can't be simultaneously covered by any translate of $2T$.

\textbf{Case 3: $T= \triangle ACD$}

The intercept along $y=0$ of any translate of $2T$ has length at most $4(\epsilon+\delta)<2$, thus $AB$ can't be covered by any translate of $2T$.

\textbf{Case 4: $T= \triangle ABC$}

The intercept along $y=1$ of any translate of $2T$ has length at most $2\epsilon<2\epsilon+2\delta$ when its 'bottom' is below $y=\epsilon-1$, thus $CD$ and $E$ can't be simultaneously covered by any translate of $2T$.

\textbf{Case 5: $T= \triangle ACE$}

The intercept along $y=0$ of any translate of $2T$ has length at most $2-2\frac{(\epsilon+\delta)(1-\epsilon)}{2-\epsilon}<2$, thus $AB$ can't be covered by any translate of $2T$.

\textbf{Case 6: $T= \triangle ADE$}

We would like to prove that any translate of $2T$ can't cover $\triangle ABC$. We extend $\overrightarrow{AD}$ by twice to $D'=(1+2\epsilon+2\delta,1)$ and $\overrightarrow{AE}$ by twice to $E'=(1,2\epsilon-2)$, then try to move $\triangle ABC$ to fit in $\triangle A'D'E'$ where $A'=A$.

Because $A'+\overrightarrow{CB}=(\epsilon+\delta,-1)$, in order for $C$ to be contained in $\triangle A'D'E'$, $B$ must lie below line
\begin{equation}\label{eq1}
    y=\frac{1}{1+\epsilon+\delta}(x-\epsilon-\delta)-1
\end{equation}

Therefore the largest possible $y$-coordinate of $B$ is that of the intersection point between line \ref{eq1} and
\begin{equation}\label{eq2}
    y=\frac{2-\epsilon}{\epsilon+\delta}(x-1-\frac{2(\epsilon+\delta)(1-\epsilon)}{2-\epsilon})
\end{equation}

By straightforward computation, we have that 
\begin{equation}\label{eq3}
    y=-\frac{2}{\frac{2-\epsilon}{\epsilon+\delta}+1-\epsilon}
\end{equation}

is the largest possible $y$-coordinate of $B$. However, the intercept along line \ref{eq3} of $\triangle A'D'E'$ equals
\begin{align*}
    &\quad (2-2\epsilon-\frac{2}{\frac{2-\epsilon}{\epsilon+\delta}+1-\epsilon})\times \frac{1+\frac{(\epsilon+\delta)(1-\epsilon)}{2-\epsilon}}{1-\epsilon}\\
    &=2+\frac{(\epsilon+\delta)(1-\epsilon)}{2-\epsilon}-\frac{2(\epsilon+\delta)}{(1-\epsilon)(2-\epsilon)}\\
    &=2-\frac{(\epsilon+\delta)(1+2\epsilon-\epsilon^2)}{(1-\epsilon)(2-\epsilon)}<2
\end{align*}
thus $C, A, B$ can't be simultaneously covered by any translate of $2T$.

For the choice of constants, any constant pair additionally satisfying $\epsilon+\delta<\frac{1}{2}$ is feasible.

\end{proof}

\section{Conclusion}
In this note, we analyze John’s theorem for simplices with positive dilation and answer related open questions raised by \cite{leme2020costly}. We prove a tight $d+2$ upper bound which matches the $d$ lower bound, improving the previously known $d^2$ bound. We also give a simple counter-example showing that the $d$ lower bound isn't optimal.

\bibliography{Xbib}
\bibliographystyle{plainnat}

\end{document}